\input ./style/arxiv-vmsta.cfg
\documentclass[numbers,compress,v1.0.1]{vmsta}
\usepackage{vtexbibtags}
\usepackage{euscript}
\usepackage{bbold,mathbh}

\volume{6}
\issue{3}
\pubyear{2019}
\firstpage{333}
\lastpage{343}
\aid{VMSTA138}
\doi{10.15559/19-VMSTA138}
\articletype{research-article}

\startlocaldefs
\newcommand{\lleft}{\left}
\newcommand{\rright}{\right}
\newtheorem{thm}{Theorem}
\newtheorem{lemma}{Lemma}
\theoremstyle{definition}
\allowdisplaybreaks
\ifdefined\HCode 
\def\index#1{}
\else 
\fi
\endlocaldefs

\begin{document}

\begin{frontmatter}
\pretitle{Research Article}

\title{On estimation of expectation of simultaneous renewal time of time-inhomogeneous Markov chains using dominating sequence}

\author{\inits{V.}\fnms{Vitaliy}~\snm{Golomoziy}\ead[label=e1]{vitaliy.golomoziy@gmail.com}}
\address{\institution{Taras Shevchenko National University of Kyiv}, Faculty~of~Mechanics~and~Mathematics,
Department~of~Probability~Theory,~Statistics~and Actuarial Mathematics,
60~Volodymyrska Street, City of Kyiv, \cny{Ukraine}, 01033}



\markboth{V. Golomoziy}{Simultaneous renewal of time-inhomogeneous Markov chains}

\begin{abstract}
The main subject of the study in this paper is the simultaneous renewal
time for two time-inhomogeneous Markov chains which start with arbitrary
initial distributions. By a simultaneous renewal we mean the first time
of joint hitting the specific set $C$ by both processes. Under the
condition of existence a dominating sequence for both renewal sequences
generated by the chains and non-lattice condition for renewal
probabilities an upper bound for the expectation of the
simultaneous renewal time is obtained.
\end{abstract}

\received{\sday{16} \smonth{3} \syear{2019}}
\revised{\sday{14} \smonth{8} \syear{2019}}
\accepted{\sday{14} \smonth{8} \syear{2019}}
\publishedonline{\sday{14} \smonth{10} \syear{2019}}
\end{frontmatter}

\section{Introduction}%
\label{sec1}

\subsection{Overview}%
\label{sec1.1}
In this paper, we are focused on studying a simultaneous renewal time
for two time-inhomogeneous, discrete-time Markov chains on a general
state space. By simultaneous renewal we mean visiting some set $C$ by
both chains\index{chains} at the same time.

While general renewal theory is well developed, properties of the
renewal sequences\index{renewal ! sequences} generated by time-inhomogeneous Markov chains are not
studied so seriously.
The literature available for this topic is scant.
Some generalizations of the general homogeneous theory could be
found, for example, in \cite{RenewalTheoryForFunctionals} or
\cite{NonlinearMarkovRenewalTheory}. At the same time, the problem under
consideration is important from both theoretical and practical
perspectives. Some of the existing results can be found in the
paper
\cite{NonHomogeneousConvolutions}, dedicated to 
investigating continuous-time
renewal processes generated by a sequence of independent but not
necessary identically distributed random variables.
The paper shows that a
general renewal equation is solvable using numerical methods and
demonstrates an application of this result in actuarial science. 
Another paper related to the renewal theory (of geometrically ergodic
time-homogeneous Markov chains) is \cite{Baxendale}.

Simultaneous renewal plays an essential role in investigation of the
stability of Markov chains,\index{Markov chains} particularly in the time-inhomogeneous case. In
the papers \cite{StabilityAndCoupling,MinorCondition} a~particular coupling construction has been
introduced to 
derive a stability estimate.\index{stability estimates} A key point of the whole
stability research is to construct a stochastic dominating \mbox{sequence} for
a series of coupling times and estimate the coupling time expectation\index{expectation} using
that sequence. Later,
the stability estimates\index{stability estimates} obtained in \cite{WidowPension} were applied to the widow pension actuarial model
which allows calculating variables related to this \mbox{actuarial} schema,\index{actuarial schema}
such as net-premiums and others. In the papers
\cite{MaxCoupling4,MaxCoupling3,MaxCoupling1,MaxCoupling2} stability estimates\index{stability estimates} are
obtained for discrete-space Markov chains in both time-homogeneous and
time-inhomogeneous case using the so-called Maximal Coupling technique.

Estimates for the expectation\index{expectation} of a simultaneous renewal time obtained
in this paper and stability estimates\index{stability estimates} derived from them have many
practical applications. In particular, we may 
point out applications in the
actuarial mathematics similar to the widow pension mentioned above.
Other examples could be: obtaining premium in actuarial models that are
represented by a perturbed Markov chain, such as seasonal factors.
Another area of application is the queuing theory. Markov chains\index{Markov chains}
stability results enable us to study non-homogeneous qeuing models or
models affected by some small non-homogeneous perturbation.

There are many works related to stability estimates\index{stability estimates} which use coupling
methods. For example stability estimates\index{stability estimates} for different versions of the
same time-homogeneous chain started with different initial
distributions, using standard coupling technique could be found in the papers
\cite{Douc,Douc2}. In \cite{MinorCondition}
we have extended results from \cite{Douc} and
\cite{Douc2} to the stability of two different (but close in some sense)
time-inhomogeneous Markov chains by introducing a modified coupling
technique which allows chains\index{chains} to couple and de-couple. Such construction
leads us to the necessity of 
investigating properties of the renewal proceses\index{renewal ! proceses}
generated by time-inhomogeneous Markov chains. Many works of other
authors are using similar coupling techniques to obtain stability
estimates\index{stability estimates} or convergence rates, see for example \cite{Fort,Fort2}
where subgeometrical Markov chains\index{subgeometrical Markov chains} are studied.

Simultaneous renewal has been studied in other works as well. In the
article \cite{Coupling0} an estimate for a coupling moment (what
we call simultaneous renewal in the present article) for
time-homogeneous processes was obtained under the existence condition
of the second moment of the renewal sequence.\index{renewal ! sequences} The critical feature of
the proof there was the Daley inequality\index{Daley inequality} (see \cite{Daley}) which
gives an estimate for average time of the excess of a renewal process
in the homogeneous case. Later, in the paper \cite{Excess} the
Daley inequality\index{Daley inequality} was extended to the time-inhomogeneous case which
allows obtaining an estimate for simultaneous renewal in the case of
a square-integrable dominating sequence for a time-inhomogeneous process
(see \cite{CouplingMoment}). This estimate involves the second
moment of the domination sequence, and in the present paper, we relax
that condition and only require the existence of the first moment. We
will compare the results of the present paper and
\cite{CouplingMoment}.

In the paper \cite{Coupling} there is a theorem that shows
finiteness of the expectation\index{expectation} of the simultaneous renewal time but
a computable estimate is not available.
There are also examples of how to check
the non-lattice regularity condition for a renewal sequence\index{renewal ! sequences} which is
involved in all the results related to the time-inhomogeneous case.

We also ask the reader to pay attention to the classical books
\cite{Lindvall,Thorisson}, devoted to the coupling method.

This paper consists of 4 sections. Section~\ref{sec1} contains introduction and
main definitions.

The main result is stated in Section~\ref{sec2}. It also includes a detailed
explanation of the conditions of the main theorem.

An example of an application of the main result concerning a birth-death
process can be found in Section~\ref{sec3}.

Section~\ref{sec4} is 
just the conclusion.

\subsection{Definitions and notations}%
\label{sec1.2}
The main object of the study is a pair of time-inhomogeneous,
independent, discrete time Markov chains\index{Markov chains} defined on the probability
space $(\varOmega , \mathfrak{F}, \mathbb{P})$, with the values in the
general state space $(E, \EuScript{E})$. We will denote these chains\index{chains} as
$X^{(1)}_{n}$ and $X^{(2)}_{n}$, $n\ge 0$, respectively.

We will use the following notation for the one-step transition
probabilities
%
\begin{align}
\begin{array}{c}
P_{lt}(x, A) = \mathbb{P}\bigl\{X^{(l)}_{t+1} \in A | X^{(l)}_{t} = x\bigr\},
\end{array} %
\end{align}
for any $x \in E$, $l \in \{1,2\}$, and any $A\in \EuScript{E}$.

Our primary goal is to obtain an upper bound for the expectation\index{expectation} of
simultaneous hitting 
a given set $C$ by both chains.\index{chains}
An application usually dictates
Applications usually dictate the choice of the set C. For a discrete state space
consisting of integers the very typical choice 
is $C=\{0\}$,
although this is not a single possible choice. Moreover, we do not
assume that the set $C$ consists of a single element. We will also allow
arbitrary initial conditions for the chains,\index{chains} and assume that each chain
has a finite expectation\index{finite expectation} for a first time hitting the set $C$ for a given
initial state.

We will continue to use the definitions and notations originated in
\cite{Coupling} and \cite{CouplingMoment}.

Define the renewal intervals
%
\begin{align}
\begin{array}{c}
\theta ^{(l)}_{0} = \inf \bigl\{t \ge 0, X^{(l)}_{t} \in C\bigr\},
\\[4pt]
\theta ^{(l)}_{n} = \inf \bigl\{t > \theta ^{(l)}_{n-1}, X^{(l)}_{t} \in C
\bigr\} - \theta ^{(l)}_{n-1},
\end{array} %
\end{align}
where $l\in \{1,2\}$, $n\ge 1$, and renewal times\index{renewal ! times}
%
\begin{align}
\tau ^{(l)}_{n} = \sum\limits
_{k=0}^{n}
\theta ^{(l)}_{k},\quad  l \in \{1,2 \},\ n\ge 0.
\end{align}

Then, for each $x\in C$, we can define the renewal probabilities
%
\begin{equation}
\label{gx_def} g^{(t,l)}_{n}(x) = \mathbb{P}\bigl\{\theta
^{(l)}_{k} = n | \tau ^{(l)}_{k-1} =
t, X^{(l)}_{t}=x \bigr\},\quad  l\in \{1,2\},\ n\ge 1,
\end{equation}
and the sequence
%
\begin{equation}
\label{gx_def_} g^{(t,l)}_{n} = \sup_{x\in C}
g^{(t,l)}_{n}(x),\quad  n\ge 1.
\end{equation}

We will also need a notation for sums of $g^{(t,l)}_{n}$:
%
\begin{equation}
\label{Gdef} G^{(t,l)}_{n} = \sum
\limits
_{k>n} g^{(t,l)}_{k}.
\end{equation}
$G^{(t,l)}_{n}$ can be interpreted as an estimate for the  probability that
renewal has not happened up to the time $n$.

The time of the simultaneous hitting the set $C$ is defined as
%
\begin{equation}
T = \inf \bigl\{ t >0:\ \exists m, n,\ t = \tau ^{(1)}_{m}
= \tau ^{(2)}_{n} \bigr\}.
\end{equation}

\section{Main result}%
\label{sec2}
First, we introduce a condition that guarantees finiteness of the
expectation\index{expectation} of the renewal times\index{renewal ! times} for both chains.\index{chains} In this paper, we use
a dominating sequence\index{dominating sequence} for this purpose.

\textit{Condition A}

\noindent There is a non-increasing sequence $G_{n},\ n\ge 0$, such that:
\begin{gather*}
G^{(t,l)}_{n} \le G_{n},
\\
\sum\limits
_{n\ge 0} G_{n} = m < \infty .
\end{gather*}

To make the further reasoning simpler, we allow the index $n$ in
$G_{n}$ to be negative. In this case we assume: $G_{n} = G_{0}$, $n<0$.

We do not require this dominating sequence\index{dominating sequence} to be probabilistic and allow
$G_{0} \ge 1$. Condition A will enable us to ascertain that both chains
$X^{(1)}, X^{(2)}$\index{chains} will have a finite expectation\index{finite expectation} of the renewal time,\index{renewal ! times}
which is a necessary condition for having a finite expectation\index{finite expectation} of
simultaneous renewal.

Opposing the paper \cite{CouplingMoment} we do not require the
existence of the second moment for a dominating sequence.\index{dominating sequence}

Next, we need some regularity condition on $u^{(t,l)}_{n}$ which
guarantees its separation from 0.

\textit{Condition B}

\noindent There is a constant $\gamma > 0$ and a number $n_{0} \ge 0$, such that for
all $t\ge 0$, $l \in \{1,2\}$, and $n \ge n_{0}$
%
\begin{equation}
\mathbb{P}\bigl\{X^{(l)}_{n+t} \in C | X^{(l)}_{n}
\in C \bigr\} \ge \gamma .
\end{equation}

It is essential that such condition also guarantees certain
``regularity'' and non-periodicity of a chain. The periodic chains\index{periodic chains} do
not satisfy it.

There are various results which allow checking the condition (A)
practically. See, for example \cite{CouplingExamples}, Theorems
4.1, 4.2, 4.3. We will use some of them later.

Before we state the main theorem, we should introduce some definitions,
which are used in the proof, and the auxiliary lemma.

The notations below are the same as in \cite{CouplingMoment}:
\begin{gather*}
\nu _{0} := \min \bigl\{j\ge 1: \tau ^{1}_{j}>n_{0}
\bigr\},
\\
B_{0} := \tau ^{1}_{\nu _{0}},
\\
\nu _{1} := \min \bigl\{j\ge \nu _{0}: \tau
^{2}_{j}-\tau ^{1}_{\nu _{0}}>n_{0},
\ \mbox{or}\ \tau ^{2}_{j} - \tau ^{1}_{\nu _{0}}=0
\bigr\},
\\
B_{1} := \tau ^{2}_{\nu _{1}} - \tau
^{1}_{\nu _{0}},
\end{gather*}
and further on,
\begin{gather*}
\nu _{2m} := \min \bigl\{j\ge \nu _{2m-1}:\ \tau
^{1}_{j} - \tau ^{2}_{\nu _{2m-1}}>n
_{0}\xch{,}{,,}\ \mbox{or}\ \tau ^{1}_{j} - \tau
^{2}_{\nu _{2m-1}}=0\bigr\},
\\
B_{2m} := \tau ^{1}_{\nu _{2m}} - \tau
^{2}_{\nu _{2m-1}},
\\
\nu _{2m+1} := \min \{j\ge \nu _{2m}:\ \tau
^{2}_{j} - \tau ^{1}_{\nu _{2m}}>n
_{0}\,,\ \mbox{or}\ \tau ^{2}_{j} - \tau
^{1}_{\nu _{2m}=0 },
\\
B_{2m+1} := \tau ^{2}_{\nu _{2m+1}} - \tau
^{1}_{\nu _{2m}}.
\end{gather*}

We will call $\nu _{k}$ 
renewal trials.\index{renewal ! trials} Let's define $\tau = \min
\{n\ge 1: B_{n} = 0\}$ and put $B_{t} = 0$, if $t\ge \tau $.

For ease of use, we will introduce the variable
\begin{equation*}
S_{n} = \sum\limits
_{k=0}^{n}
B_{k}.
\end{equation*}
Note, that $S_{2n} = \tau ^{1}_{\nu _{2n}}$, $S_{2n+1} = \tau ^{2}_{\nu
_{2n+1}}$, however, the use of the variable $S_{n}$ makes the text much more
readable.

\begin{thm}
\label{theorem1}
Assume that chains\index{chains} $(X^{(l)}_{n})$, $l\in \{1,2\}$, start with some
initial distributions $\lambda _{1}$ and $\lambda _{2}$ and that
expectations\index{expectation} of the first time each process hits the set $C$ are finite.
Denote them as $m_{1}(\lambda _{1})$ and $m_{2}(\lambda _{2})$
respectively. Assume also, that conditions (A) and (B) introduced above
hold true. In this case, there exists a dominating sequence
$\hat{S}_{n}$\index{dominating sequence} for the simultaneous renewal time $T$:
\begin{equation*}
\mathbb{P}(T>n) \le \hat{S}_{n},
\end{equation*}
such that
%
\begin{equation}
\label{main_result} \mathbb{E}[T] \le \sum\limits
_{n \ge 0}
\hat{S}_{n} \le m_{1}(\lambda _{1}) +
m_{2}(\lambda _{2}) + (n_{0}
G_{0} + m) (1+\gamma )/\gamma .
\end{equation}
\end{thm}
\begin{proof}
First, we assume that $\theta ^{(2)}_{0} = 0$, which means that the
second chain starts from the set $C$.

Next, the upper bound for $T$ was introduced in
\cite{CouplingMoment} (it trivially follows from definitions):
%
\begin{equation}
\label{T_ineq} T \le \theta ^{(1)}_{0} + \sum
_{n=0}^{\tau }B_{n} = \theta
^{(1)}_{0} + \sum_{n\ge 0}
B_{n} \mathbh{1}_{\tau > n}.
\end{equation}

We introduce the variable
\begin{equation*}
T^{\prime }= \sum_{n=0}^{\tau }B_{n}.
\end{equation*}
Then $\theta ^{(1)}_{0} + T^{\prime }$ is pointwise bigger than $T$, and so
it stochastically dominates $T$.

So we will focus on building a dominating sequence\index{dominating sequence} for $T^{\prime }$ and
estimating its expectation.\index{expectation}

Let us consider
\begin{equation*}
\mathbb{P}\bigl\{T^{\prime }> n\bigr\} = \sum\limits
_{k=0}^{n} \mathbb{P}\{S _{k} \le n <
S_{k+1}\} = \sum\limits
_{k=0}^{n} \sum
\limits
_{j=k}^{n} \mathbb{P}\{S_{k}=j,
B_{k+1} > n-j\}.
\end{equation*}
Note that some terms are equal to $0$ in the latter sum because each
renewal trial\index{renewal ! trials} takes at least $n_{0}$ steps, but this fact does not
affect our reasoning.

Lemma~\ref{lem1} gives us the inequality
\begin{equation*}
\mathbb{P}\{S_{k}=j, B_{k+1}>n-j\} \le \mathbb{P}
\{S_{k}=j,\tau \ge k\} G_{n-j-n_{0}},
\end{equation*}
which allows us to build the upper bound for $\mathbb{P}\{T^{\prime }>
n\}$:
\begin{align*}
& \mathbb{P}\bigl\{T^{\prime }> n\bigr\} \le \sum\limits
_{k=0}^{n}\sum\limits
_{j=k} ^{n}
\mathbb{P}\{S_{k}=j, \tau \ge k \} G_{n-j-n_{0}}
\\
&\quad =\sum\limits
_{j=0}^{n} G_{n-j-n_{0}} \sum
\limits
_{k=0}^{j} P(S_{k}=j, \tau \ge k).
\end{align*}
So we can put
\begin{equation*}
\hat{S}^{\prime }_{n} = \sum\limits
_{j=0}^{n} G_{n-j-n_{0}} \sum
\limits
_{k=0}^{j} P(S_{k}=j,\tau \ge k),
\end{equation*}
and so the dominating sequence\index{dominating sequence} for $\mathbb{P}\{T^{\prime }> n\}$ is
built.

Let's now estimate its expectation:\index{expectation}
\begin{align*}
\mathbb{E}\bigl[T^{\prime }\bigr] &= \sum\limits
_{n\ge 0}
\mathbb{P}\bigl\{T^{\prime
}> n\bigr\} \le \sum\limits
_{n\ge 0}
\hat{S}^{\prime }_{n}
\\
&=\sum\limits
_{n\ge 0} \Biggl(\sum\limits
_{j=0}^{n}
G_{n-j-n_{0}} \sum\limits
_{k=0}^{j}
P(S_{k}=j,\tau \ge k) \Biggr)
\\
&=\biggl(\sum\limits
_{n\ge 0}G_{n-n_{0}} \biggr) \times
\biggl(\sum\limits
_{n\ge 0}\sum\limits
_{k\ge n}
\mathbb{P}\{S_{n}=k, \tau \ge n \} \biggr)
\\
&=(n_{0}G_{0} + m) \sum\limits
_{n\ge 0}
\mathbb{P}\{\tau \ge n\}.
\end{align*}

Lemma 8.5 from \cite{CouplingMoment} gives us the inequality
\begin{equation*}
\mathbb{P}\{\tau > n\} \le (1-\gamma )^{n}
\end{equation*}
which yields
%
\begin{equation}
\mathbb{E}\bigl[T^{\prime }\bigr] \le (n_{0}
G_{0} + m) (1 + 1/\gamma ) = (n_{0} G_{0}
+ m) (1 + \gamma )/\gamma .
\end{equation}

Now, we have to get rid of the assumption $\theta ^{(2)}_{0}=0$.
Following the same calculations as in \cite{Coupling} after
formula (\ref{eq20}) we may estimate $T$:
\begin{equation*}
T\le \theta ^{(1)}_{0} + \theta ^{(2)}_{0}
+ T^{\prime },
\end{equation*}
which entails the existence of a dominating sequence $\hat{S}_{n}
\ge \mathbb{P}\{T>n\}$\index{dominating sequence} and the estimate
\begin{equation*}
\mathbb{E}[T] \le \sum\limits
_{n\ge 0} \hat{S}_{n} \le
m_{1}(\lambda _{1}) + m_{2}(\lambda
_{2}) + (n_{0} G_{0} + m) (1+\gamma )/
\gamma .\qedhere
\end{equation*}
\end{proof}

\begin{lemma}\label{lem1}
Assume that the chain $X^{(2)}$\index{chains} starts from $x\in C$. Then, for all
$k,j,t\ge 0$ the following inequality holds true:
\begin{equation*}
\mathbb{P}\{S_{k} = j, B_{k+1} > t\} \le
P(S_{k} = j, \tau \ge k) G _{t-n_{0}}.
\end{equation*}
\end{lemma}
\begin{proof}
We start with the assumption that $k$ is even. We can write then:
\begin{equation*}
\mathbb{P}\{S_{k}=j, B_{k+1} > t\} = \mathbb{P}\bigl
\{S_{k} = j, X^{(2)} _{n} \notin C, n\in
\{j,j+n_{0},\ldots ,j+t\}\bigr\}.
\end{equation*}
Let's denote as $\eta $ the last renewal of the chain $X^{(2)}$\index{chains} before
the time $j+n_{0}$. By the construction of the sequence $B_{k}$,
\begin{equation*}
\eta \ge S_{k}-B_{k} = S_{k-1}.
\end{equation*}
Indeed, if $k>0$, this means that $S_{k-1}=S_{k}-B_{k}>0$ is a renewal
time\index{renewal ! times} for the chain $X^{(2)}$. If $k=0$, then we should put $S_{-1}=S
_{0}-B_{0}=0$, but the chain $X^{(2)}$\index{chains} has started in $C$, so
$\eta > 0$ by the condition of the lemma. So by construction, we have:
\begin{equation*}
S_{k} - B_{k} \le \eta < j + n_{0}.
\end{equation*}
Let us denote the 
sets
\begin{equation*}
A_{l}(s,t) = \bigl\{X^{(l)}_{m} \notin C, m
\in \{s,\ldots , t\}\bigr\},\quad  l \in \{1,2\},
\end{equation*}
and inspect the sample path of the chain $X^{(2)}$\index{chains} after the moment
$\eta $:
%
\begin{align}
\mathbb{P}\{S_{k}=j, B_{k+1} > t\} &= \mathbb{P}\bigl\{S_{k} = j, X^{(2)}
_{j} \notin C, A_{2}(j+n_{0}, j+t)\bigr\}\nonumber
\\
& =\sum\limits _{u, v} \mathbb{P}\bigl\{S_{k-1} = u, \tau \ge k,
A_{1}(u+n_{0},j-1), \nonumber\\
&\qquad\qquad  X^{(1)}_{j}\in C, X^{(2)}_{v} \in C, A_{2}(v+1,j+t)
\bigr\}, \label{lemma1_0} %
\end{align}
where $u<j$, and $v\ge u$ is a value of $\eta $. Note that
\begin{equation*}
\bigl\{S_{k-1}=u, X^{(1)}_{u} \notin C\bigr\}
= \{S_{k-1}=u, \tau >k-1\} = \{S _{k-1}=u, \tau \ge k\}.
\end{equation*}
Let us recall that the chains\index{chains} $X^{(1)}$ and $X^{(2)}$ are independent and
so we can split the probability in the formula (\ref{lemma1_0}) into the
product:
\begin{align*}
&\mathbb{P}\bigl\{S_{k-1} = u,\tau \ge k, A_{1}(u+n_{0},j-1),
X^{(1)}_{j} \in C, X^{(2)}_{v}\in
C,\ A_{2}(v+1,j+t) \bigr\}
\\
&\quad =\int_{E\times C} \mathbb{P}\bigl\{S_{k-1}=u,
A_{1}(u+n_{0},v), \bigl(X^{(1)}
_{v}, X^{(2)}_{v}\bigr)=(dx,dy)\bigr\}
\\
&\qquad \times\mathbb{P}\bigl\{A_{1}(v+1, j-1), X^{(1)}_{j}
\in C, A_{2}(v+1, j+t)| \bigl(X ^{(1)}_{v},X^{(2)}_{v}
\bigr)=(x,y)\bigr\}
\\
&\quad =\int_{E\times C} \mathbb{P}\bigl\{S_{k-1}=u,\tau \ge
k, A_{1}(u+n_{0},v), \bigl(X ^{(1)}_{v},
X^{(2)}_{v}\bigr)=(dx,dy)\bigr\}
\\
&\qquad \times\mathbb{P}\bigl\{A_{1}(v+1, j-1), X^{(1)}_{j}
\in C|X^{(1)}_{v}=x\bigr\} \mathbb{P}\bigl\{
A_{2}(v+1, j+t)|X^{(2)}_{v}=y\bigr\}.
\end{align*}
But $\mathbb{P}\{A_{2}(v+1,j+t)|X^{(2)}_{v}=y\}$, $y\in C$, is 
just the
probability that the next renewal \xch{after}{afer} $v$ will not happen until the moment
$j+t$. So we have
%
\begin{equation}
\label{lemma1_1} \sup_{y\in C} \mathbb{P}\bigl
\{A_{2}(v+1,j+t)|X^{(2)}_{v}=y\bigr\} =
G^{(v,2)} _{j+t-v} \le G_{j+t-v},
\end{equation}
where the latest inequality follows from the Condition A.

The above reasoning is valid in the case $\eta < j$, however, the case
$\eta \in \{j+1,j+n_{0}-1\}$ yields the same result. The maximal
possible value of $v$ is $j+n_{0}$, so, using the fact that sequence
$G_{n}$ is non-increasing, it follows from (\ref{lemma1_1}) that
%
\begin{equation}
\label{lemma1_2} \sup_{y\in C} \mathbb{P}\bigl
\{A_{2}(v+1,j+t)|X^{(2)}_{v}=y\bigr\} \le
G_{j+t-v} \le G_{t-n_{0}},
\end{equation}
which gives the estimate that does not depend on $v$. Taking into
account the inequality (\ref{lemma1_2}) we can derive
%
\begin{align}
&\mathbb{P}\{S_{k}=j, B_{k+1} > t\}\nonumber
\\
&\quad \le\sum\limits _{u,v} \int_{E\times C} \mathbb{P}\bigl\{S_{k-1}=u, \tau \ge k,A
_{1}(u+n_{0},v), (X^{(1)}_{v}, X^{(2)}_{v})=(dx,dy)\bigr\}\nonumber
\\
&\qquad \times\mathbb{P}\bigl\{A_{1}(v+1, j-1), X^{(1)}_{j}\in C|X^{(1)}_{v}=x\bigr\} G_{t-n
_{0}} \le \mathbb{P}\{S_{k}=j,\tau \ge k\} G_{t-n_{0}}.
\end{align}

To complete the proof, we should consider the case when $k$ is odd. In
this case, $k>0$ which means $k-1\ge 0$ and we don't have a problem with
$k=0$ like in the even case. So, similar reasoning will yield the same
inequality for odd $k$.
\end{proof}

\section{Application to the birth-death processes}%
\label{sec3}
In the paper \cite{CouplingMoment} we had derived an estimate for
a simultaneous renewal of two time-inhomogeneous birth-death processes
$X^{(1)}$ and $X^{(2)}$ with the following transition probabilities on
the $t$-th step:
%
\begin{equation}
P_{t} = \lleft ( %
\begin{array}{cccccc}
\alpha _{t0}& 1-\alpha _{t0}& 0& 0& 0&\ldots
\\
0& \alpha _{t1}& 0& 1-\alpha _{t1}& 0& \ldots
\\
0& 0& \alpha _{t2}& 0& 1-\alpha _{t2}& \ldots
\\
\ldots
\end{array} %
\rright )
\end{equation}
and
%
\begin{equation}
Q_{t} = \lleft ( %
\begin{array}{cccccc}
\beta _{t0}& 1-\beta _{t0}& 0& 0& 0&\ldots
\\
0& \beta _{t1}& 0& 1-\beta _{t1}& 0& \ldots
\\
0& 0& \beta _{t2}& 0& 1-\beta _{t2}& \ldots
\\
\ldots
\end{array} %
\rright )
\end{equation}

The set $C$ is equal to $0$:
\begin{equation*}
C=\{0\}.
\end{equation*}

An estimate from \cite{CouplingMoment} involves the second moment
of the dominating sequence.\index{dominating sequence} Let us compare an estimate from
\cite{CouplingMoment} with the one that follows from  Theorem~\ref{theorem1}.

We will start with building $\gamma $. As in
\cite{CouplingMoment} we notice that for every $t>0$, $g^{(l)}_{1} =
\alpha _{t0} > 0$, and assume
%
\begin{equation}
\label{gamma0_cond} \gamma _{0} = \inf_{t} \{
\alpha _{t0}, \beta _{t0}\} > 0.
\end{equation}

As it has been shown in the \cite{CouplingMoment} $\gamma $ can
be chosen in the following way:
%
\begin{equation}
\label{gamma_def} \gamma = \exp \bigl(\hat{\mu }ln(\gamma _{0})/
\gamma _{0} \bigr) = \gamma _{0}^{\hat{\mu }/\gamma _{0}}
\end{equation}
with the $n_{0}=0$.

We will use the same dominating sequence\index{dominating sequence} as in
\cite{CouplingMoment}. To do this, we consider a standard random walk
with a parameter $p>1/2$ and a probability $f_{n}$ of a first return to
$0$ in $n$ steps. It has been shown in \cite{Feller}, Chapter
XIII, that the generating function $F(z)$ for $f_{n}$ looks like
%
\begin{equation}\label{eq20}
F(z) = \frac{1-\sqrt{1-4p(1-p)s^{2}}}{2(1-p)}.
\end{equation}
We have shown in the \cite{CouplingMoment} that $G_{n}=\sum\limits _{k>n} f_{n}/p$ is a dominating sequence\index{dominating sequence} (non-probabilistic) for
renewal sequences\index{renewal ! sequences} generated by $X^{(1)}$ and $X^{(2)}$, if
%
\begin{equation}
p(1-p)\ge \xch{\sup}{sup}_{t,j}\bigl\{\alpha _{tj}(1-\alpha
_{tj}),\beta _{tj}(1-\beta _{tj})\bigr\}
\end{equation}

The estimate for the expectation\index{expectation} of the simultaneous renewal time in
case of both chains started at $C$ looks like
%
\begin{equation}
E_{1} = \mathbb{E}[T] \le \hat{\mu }_{2}/\gamma + \hat{
\mu }_{1}/ \gamma ^{2},
\end{equation}
where
%
\begin{equation}
\begin{aligned}
\hat{\mu }_{1}&= 2/(2p-1) + 1 ,
\\
\hat{\mu }_{2} = (2p-1)^{-1}  &\biggl( 2 + \frac{8(1-p)}{1-4p}  \biggr) +
2/(2p-1) + 1.
\end{aligned}
\end{equation}
In our case the estimate looks like
\begin{equation*}
E_{2} = \mathbb{E}[T] \le \hat{\mu }_{1}(1+\gamma )/
\gamma .
\end{equation*}

So, we can see that
\begin{equation*}
E_{2} = (E_{1} - \hat{\mu }_{2}/\gamma ) (1+
\gamma )\gamma ,
\end{equation*}
for all $\gamma \in (0,\frac{\sqrt{5}-1}{2})$, $\gamma (1+\gamma ) <
1$ which means that $E_{2}$ is a better estimate than $E_{1}$. In the
same time, $\gamma _{0}$ is typically a small number, which makes
$\gamma $ to be a very small number. So, the improvement in estimate can
be significant for practical applications.

\section{Conclusions}%
\label{sec4}
In this article, we obtained an estimate for the expectation\index{expectation} of the
simultaneous renewal time for two time-inhomogeneous, discrete-time
Markov chains on a general state space. By the simultaneous renewal of
two Markov chains\index{Markov chains} $X$ and $X^{\prime }$ we understand the first moment
of hitting a specific set $C$ by both chains.\index{chains}

We have shown that under conditions (A) and (B) an estimate has the form
%
\begin{equation}
\mathbb{E}[T] \le \sum\limits
_{n \ge 0} \hat{S}_{n} \le
m_{1}(\lambda _{1}) + m_{2}(\lambda
_{2}) + (n_{0} G_{0} + m) (1+\gamma )/
\gamma .
\end{equation}

We have shown how the parameters $\gamma , n_{0}, G_{0}, m_{1}(\lambda
_{1}), m_{2}(\lambda _{2})$, and the sequence $\hat{S}_{n}$ included in the
estimate can be calculated.

This result allows us to refine the stability estimate\index{stability estimates} for two
time-inhomogeneous Markov chains, which is a subject of a further
investigation.


\begin{appendix}
\appendix
\end{appendix}




\begin{thebibliography}{99}

\bibitem{Fort}
\begin{barticle}
\bauthor{\bsnm{Andrieu}, \binits{C.}},
\bauthor{\bsnm{Fort}, \binits{G.}},
\bauthor{\bsnm{Vihola}, \binits{M.}}:
\batitle{Quantitative convergence rates for sugeometric {Markov} chains}.
\bjtitle{Ann. Appl. Probab.}
\bvolume{52},
\bfpage{391}--\blpage{404}
(\byear{2015})
\bid{doi={10.1239/jap/\\1437658605}, mr={3372082}}
\end{barticle}
%
\OrigBibText
\begin{barticle}
\bauthor{\bsnm{Andrieu}, \binits{C.}},
\bauthor{\bsnm{Fort}, \binits{G.}},
\bauthor{\bsnm{Vihola}, \binits{M.}}:
\batitle{Quantitative convergence rates for sugeometric {Markov} chains}.
\bjtitle{Annals of Applied Probability}
\bvolume{52},
\bfpage{391}--\blpage{404}
(\byear{2015})
\end{barticle}
\endOrigBibText
\bptok{structpyb}%
\endbibitem

\bibitem{Baxendale}
\begin{barticle}
\bauthor{\bsnm{Baxendale}, \binits{P.}}:
\batitle{Renewal theory and computable convergence rates for geometrically
 ergodic {Markov} chains}.
\bjtitle{Ann. Appl. Probab.}
\bvolume{15},
\bfpage{700}--\blpage{738}
(\byear{2005})
\bid{doi={10.1214/105051604000000710}, mr={2114987}}
\end{barticle}
%
\OrigBibText
\begin{barticle}
\bauthor{\bsnm{Baxendale}, \binits{P.}}:
\batitle{Renewal theory and computable convergence rates for geometrically
 ergodic {Markov} chains}.
\bjtitle{Annals of Applied Probability}
\bvolume{15},
\bfpage{700}--\blpage{738}
(\byear{2005})
\end{barticle}
\endOrigBibText
\bptok{structpyb}%
\endbibitem

\bibitem{Daley}
\begin{barticle}
\bauthor{\bsnm{Daley}, \binits{D.}}:
\batitle{Tight bounds for the renewal function of a random walk}.
\bjtitle{Ann. Probab.}
\bvolume{8},
\bfpage{615}--\blpage{621}
(\byear{1980})
\bid{mr={0573298}}
\end{barticle}
%
\OrigBibText
\begin{barticle}
\bauthor{\bsnm{Daley}, \binits{D.}}:
\batitle{Tight bounds for the renewal function of a random walk}.
\bjtitle{Ann.Probab.}
\bvolume{8},
\bfpage{615}--\blpage{621}
(\byear{1980})
\end{barticle}
\endOrigBibText
\bptok{structpyb}%
\endbibitem

\bibitem{Douc}
\begin{barticle}
\bauthor{\bsnm{Douc}, \binits{R.}},
\bauthor{\bsnm{Moulines}, \binits{E.}},
\bauthor{\bsnm{Soulier}, \binits{P.}}:
\batitle{Practical drift conditions for subgeometric rates of convergence}.
\bjtitle{Ann. Appl. Probab.}
\bvolume{14},
\bfpage{1353}--\blpage{1377}
(\byear{2004})
\bid{doi={\\10.1214/105051604000000323}, mr={2071426}}
\end{barticle}
%
\OrigBibText
\begin{barticle}
\bauthor{\bsnm{Douc}, \binits{R.}},
\bauthor{\bsnm{Moulines}, \binits{E.}},
\bauthor{\bsnm{Soulier}, \binits{P.}}:
\batitle{Practical drift conditions for subgeometric rates of convergence}.
\bjtitle{Annals of Applied Probability}
\bvolume{14},
\bfpage{1353}--\blpage{1377}
(\byear{2004})
\end{barticle}
\endOrigBibText
\bptok{structpyb}%
\endbibitem

\bibitem{Douc2}
\begin{barticle}
\bauthor{\bsnm{Douc}, \binits{R.}},
\bauthor{\bsnm{Moulines}, \binits{E.}},
\bauthor{\bsnm{Soulier}, \binits{P.}}:
\batitle{Quantitative bounds on convergence of time-inhomogeneous {Markov}
 chains}.
\bjtitle{Ann. Appl. Probab.}
\bvolume{14},
\bfpage{1643}--\blpage{1665}
(\byear{2004})
\bid{doi={10.1214/105051604000000620}, mr={2099647}}
\end{barticle}
%
\OrigBibText
\begin{barticle}
\bauthor{\bsnm{Douc}, \binits{R.}},
\bauthor{\bsnm{Moulines}, \binits{E.}},
\bauthor{\bsnm{Soulier}, \binits{P.}}:
\batitle{Quantitative bounds on convergence of time-inhomogeneous {Markov}
 chains}.
\bjtitle{Annals of Applied Probability}
\bvolume{14},
\bfpage{1643}--\blpage{1665}
(\byear{2004})
\end{barticle}
\endOrigBibText
\bptok{structpyb}%
\endbibitem

\bibitem{Feller}
\begin{bbook}
\bauthor{\bsnm{Feller}, \binits{W.}}:
\bbtitle{An Introduction to Probability Theory and Its Applications, Vol. 1}.
\bpublisher{John Wiley and Sons}
(\byear{1957})
\bid{mr={0088081}}
\end{bbook}
%
\OrigBibText
\begin{bbook}
\bauthor{\bsnm{Feller}, \binits{W.}}:
\bbtitle{An Introduction to Probability Theory and Its Applications, Vol. 1}.
\bpublisher{John Wiley and Sons}
(\byear{1957})
\end{bbook}
\endOrigBibText
\bptok{structpyb}%
\endbibitem

\bibitem{Fort2}
\begin{barticle}
\bauthor{\bsnm{Fort}, \binits{G.}},
\bauthor{\bsnm{Roberts}, \binits{G.O.}}:
\batitle{Subgeometric ergodicity of strong {Markov} processes}.
\bjtitle{Ann. Appl. Probab.}
\bvolume{15},
\bfpage{1565}--\blpage{1589}
(\byear{2005})
\bid{doi={10.1214/\\105051605000000115}, mr={2134115}}
\end{barticle}
%
\OrigBibText
\begin{barticle}
\bauthor{\bsnm{Fort}, \binits{G.}},
\bauthor{\bsnm{Roberts}, \binits{G.O.}}:
\batitle{Subgeometric ergodicity of strong {Markov} processes}.
\bjtitle{Annals of Applied Probability}
\bvolume{15},
\bfpage{1565}--\blpage{1589}
(\byear{2005})
\end{barticle}
\endOrigBibText
\bptok{structpyb}%
\endbibitem

\bibitem{NonHomogeneousConvolutions}
\begin{barticle}
\bauthor{\bsnm{Gismondi}, \binits{F.}},
\bauthor{\bsnm{Janssen}, \binits{J.}},
\bauthor{\bsnm{R.}, \binits{M.}}:
\batitle{Non-homogeneous time convolutions, renewal processes and age-dependent
 mean number of notorcar accidents}.
\bjtitle{Ann. Actuar. Sci.}
\bvolume{9},
\bfpage{36}--\blpage{57}
(\byear{2015})
\end{barticle}
%
\OrigBibText
\begin{barticle}
\bauthor{\bsnm{Gismondi}, \binits{F.}},
\bauthor{\bsnm{Janssen}, \binits{J.}},
\bauthor{\bsnm{R.}, \binits{M.}}:
\batitle{Non-homogeneous time convolutions, renewal processes and age-dependent
 mean number of notorcar accidents}.
\bjtitle{Annals of Actuarial Science}
\bvolume{9},
\bfpage{36}--\blpage{57}
(\byear{2015})
\end{barticle}
\endOrigBibText
\bptok{structpyb}%
\endbibitem

\bibitem{StabilityAndCoupling}
\begin{barticle}
\bauthor{\bsnm{Golomoziy}, \binits{V.}}:
\batitle{A subgeometric estimate of the stability for time-homogeneous {Markov}
 chains}.
\bjtitle{Theory Probab. Math. Stat.}
\bvolume{81},
\bfpage{35}--\blpage{50}
(\byear{2010})
\bid{doi={\\10.1090/S0094-9000-2010-00808-8}, mr={2667308}}
\end{barticle}
%
\OrigBibText
\begin{barticle}
\bauthor{\bsnm{Golomoziy}, \binits{V.}}:
\batitle{A subgeometric estimate of the stability for time-homogeneous {Markov}
 chains}.
\bjtitle{Theory of probability and mathematical statistics}
\bvolume{81},
\bfpage{35}--\blpage{50}
(\byear{2010})
\end{barticle}
\endOrigBibText
\bptok{structpyb}%
\endbibitem

\bibitem{MinorCondition}
\begin{barticle}
\bauthor{\bsnm{Golomoziy}, \binits{V.}}:
\batitle{An estimate of the stability for nonhomogeneous {Markov} chains under
 classical minorization condition}.
\bjtitle{Theory Probab. Math. Stat.}
\bvolume{88},
\bfpage{35}--\blpage{49}
(\byear{2014})
\bid{doi={10.1090/S0094-9000-2014-00917-5}, mr={3112633}}\vadjust{\vfill\goodbreak}
\end{barticle}
%
\OrigBibText
\begin{barticle}
\bauthor{\bsnm{Golomoziy}, \binits{V.}}:
\batitle{An estimate of the stability for nonhomogeneous {Markov} chains under
 classical minorization condition}.
\bjtitle{Theory of probability and mathematical statistics}
\bvolume{88},
\bfpage{35}--\blpage{49}
(\byear{2014})
\end{barticle}
\endOrigBibText
\bptok{structpyb}%
\endbibitem

\bibitem{CouplingExamples}
\begin{barticle}
\bauthor{\bsnm{Golomoziy}, \binits{V.}}:
\batitle{An inequality for the coupling moment in the case of two inhomogeneous
 {Markov} chains}.
\bjtitle{Theory Probab. Math. Stat.}
\bvolume{90},
\bfpage{43}--\blpage{56}
(\byear{2015})
\bid{doi={10.1090/tpms/948}, mr={3241859}}
\end{barticle}
%
\OrigBibText
\begin{barticle}
\bauthor{\bsnm{Golomoziy}, \binits{V.}}:
\batitle{An inequality for the coupling moment in the case of two inhomogeneous
 {Markov} chains}.
\bjtitle{Theory of probability and mathematical statistics}
\bvolume{90},
\bfpage{43}--\blpage{56}
(\byear{2015})
\end{barticle}
\endOrigBibText
\bptok{structpyb}%
\endbibitem

\bibitem{CouplingMoment}
\begin{barticle}
\bauthor{\bsnm{Golomoziy}, \binits{V.}}:
\batitle{An estimate for an expectation of the simultaneous renewal for
 time-inhomogeneous {Markov} chains}.
\bjtitle{Mod. Stoch. Theory Appl.}
\bfpage{315}--\blpage{323}
(\byear{2016})
\bid{doi={10.15559/16-VMSTA68}, mr={3593115}}
\end{barticle}
%
\OrigBibText
\begin{botherref}
\oauthor{\bsnm{Golomoziy}, \binits{V.}}:
An estimate for an expectation of the simultaneous renewal for
 time-inhomogeneous {Markov} chains.
Modern Stochastics: Theory and Applications,
315--323
(2016)
\end{botherref}
\endOrigBibText
\bptok{structpyb}%
\endbibitem

\bibitem{Excess}
\begin{barticle}
\bauthor{\bsnm{Golomoziy}, \binits{V.}}:
\batitle{An estimate of the expectation of the excess of a renewal sequence
 generated by a time-inhomogeneous {Markov} chain if a square-integrable
 majorizing squence exists}.
\bjtitle{Theory Probab. Math. Stat.}
\bvolume{94},
\bfpage{53}--\blpage{62}
(\byear{2017})
\bid{doi={\\10.1090/tpms/1008}, mr={3553453}}
\end{barticle}
%
\OrigBibText
\begin{barticle}
\bauthor{\bsnm{Golomoziy}, \binits{V.}}:
\batitle{An estimate of the expectation of the excess of a renewal sequence
 generated by a time-inhomogeneous {Markov} chain if a square-integrable
 majorizing squence exists}.
\bjtitle{Theory of probability and mathematical statistics}
\bvolume{94},
\bfpage{53}--\blpage{62}
(\byear{2017})
\end{barticle}
\endOrigBibText
\bptok{structpyb}%
\endbibitem

\bibitem{Coupling}
\begin{barticle}
\bauthor{\bsnm{Golomoziy}, \binits{V.}},
\bauthor{\bsnm{Kartashov}, \binits{M.}}:
\batitle{On the integrability of the coupling moment for time-inhomogeneous
 {Markov} chains}.
\bjtitle{Theory Probab. Math. Stat.}
\bvolume{89},
\bfpage{1}--\blpage{12}
(\byear{2014})
\bid{doi={10.1090/S0094-9000-2015-00930-3}, mr={3235170}}
\end{barticle}
%
\OrigBibText
\begin{barticle}
\bauthor{\bsnm{Golomoziy}, \binits{V.}},
\bauthor{\bsnm{Kartashov}, \binits{M.}}:
\batitle{On the integrability of the coupling moment for time-inhomogeneous
 {Markov} chains}.
\bjtitle{Theory of probability and mathematical statistics}
\bvolume{89},
\bfpage{1}--\blpage{12}
(\byear{2014})
\end{barticle}
\endOrigBibText
\bptok{structpyb}%
\endbibitem

\bibitem{MaxCoupling4}
\begin{barticle}
\bauthor{\bsnm{Golomoziy}, \binits{V.}},
\bauthor{\bsnm{Kartashov}, \binits{M.}}:
\batitle{Maximal coupling and stability of discrete non-homogeneous {Markov}
 chains}.
\bjtitle{Theory Probab. Math. Stat.}
\bvolume{91},
\bfpage{17}--\blpage{27}
(\byear{2015})
\bid{doi={10.1090/S0094-9000-2013-00891-6}, mr={2986452}}
\end{barticle}
%
\OrigBibText
\begin{barticle}
\bauthor{\bsnm{Golomoziy}, \binits{V.}},
\bauthor{\bsnm{Kartashov}, \binits{M.}}:
\batitle{Maximal coupling and stability of discrete non-homogeneous {Markov}
 chains}.
\bjtitle{Theory of probability and mathematical statistics}
\bvolume{91},
\bfpage{17}--\blpage{27}
(\byear{2015})
\end{barticle}
\endOrigBibText
\bptok{structpyb}%
\endbibitem

\bibitem{MaxCoupling3}
\begin{barticle}
\bauthor{\bsnm{Golomoziy}, \binits{V.}},
\bauthor{\bsnm{Kartashov}, \binits{M.}}:
\batitle{Maxmimal coupling and v-stability of discrete nonhomogeneous {Markov}
 chains}.
\bjtitle{Theory Probab. Math. Stat.}
\bvolume{93},
\bfpage{19}--\blpage{31}
(\byear{2016})
\bid{doi={10.1090/tpms/992}, mr={3553437}}
\end{barticle}
%
\OrigBibText
\begin{barticle}
\bauthor{\bsnm{Golomoziy}, \binits{V.}},
\bauthor{\bsnm{Kartashov}, \binits{M.}}:
\batitle{Maxmimal coupling and v-stability of discrete nonhomogeneous {Markov}
 chains}.
\bjtitle{Theory of probability and mathematical statistics}
\bvolume{93},
\bfpage{19}--\blpage{31}
(\byear{2016})
\end{barticle}
\endOrigBibText
\bptok{structpyb}%
\endbibitem

\bibitem{WidowPension}
\begin{barticle}
\bauthor{\bsnm{Golomoziy}, \binits{V.}},
\bauthor{\bsnm{Kartashov}, \binits{M.}},
\bauthor{\bsnm{Kartashov}, \binits{Y.}}:
\batitle{Impact of the stress factor on the price of widow's pensions. proofs}.
\bjtitle{Theory Probab. Math. Stat.}
\bvolume{92},
\bfpage{17}--\blpage{22}
(\byear{2016})
\bid{mr={3330687}}
\end{barticle}
%
\OrigBibText
\begin{barticle}
\bauthor{\bsnm{Golomoziy}, \binits{V.}},
\bauthor{\bsnm{Kartashov}, \binits{M.}},
\bauthor{\bsnm{Kartashov}, \binits{Y.}}:
\batitle{Impact of the stress factor on the price of widow's pensions. proofs}.
\bjtitle{Theory of probability and mathematical statistics}
\bvolume{92},
\bfpage{17}--\blpage{22}
(\byear{2016})
\end{barticle}
\endOrigBibText
\bptok{structpyb}%
\endbibitem

\bibitem{Coupling0}
\begin{barticle}
\bauthor{\bsnm{Kartashov}, \binits{M.}},
\bauthor{\bsnm{Golomoziy}, \binits{V.}}:
\batitle{Average coupling time for independent discrete renewal processes}.
\bjtitle{Theory Probab. Math. Stat.}
\bvolume{84},
\bfpage{77}--\blpage{83}
(\byear{2011})
\bid{doi={\\10.1090/S0094-9000-2012-00855-7}, mr={2857418}}
\end{barticle}
%
\OrigBibText
\begin{barticle}
\bauthor{\bsnm{Kartashov}, \binits{M.}},
\bauthor{\bsnm{Golomoziy}, \binits{V.}}:
\batitle{Average coupling time for independent discrete renewal processes}.
\bjtitle{Theory of probability and mathematical statistics}
\bvolume{84},
\bfpage{77}--\blpage{83}
(\byear{2011})
\end{barticle}
\endOrigBibText
\bptok{structpyb}%
\endbibitem

\bibitem{MaxCoupling1}
\begin{barticle}
\bauthor{\bsnm{Kartashov}, \binits{M.}},
\bauthor{\bsnm{Golomoziy}, \binits{V.}}:
\batitle{Maximal coupling procedure and stability of discrete {Markov} chains.
 i}.
\bjtitle{Theory Probab. Math. Stat.}
\bvolume{86},
\bfpage{93}--\blpage{104}
(\byear{2013})
\bid{doi={10.1090/S0094-9000-2014-00905-9}, mr={3241447}}
\end{barticle}
%
\OrigBibText
\begin{barticle}
\bauthor{\bsnm{Kartashov}, \binits{M.}},
\bauthor{\bsnm{Golomoziy}, \binits{V.}}:
\batitle{Maximal coupling procedure and stability of discrete {Markov} chains.
 i}.
\bjtitle{Theory of probability and mathematical statistics}
\bvolume{86},
\bfpage{93}--\blpage{104}
(\byear{2013})
\end{barticle}
\endOrigBibText
\bptok{structpyb}%
\endbibitem

\bibitem{MaxCoupling2}
\begin{barticle}
\bauthor{\bsnm{Kartashov}, \binits{M.}},
\bauthor{\bsnm{Golomoziy}, \binits{V.}}:
\batitle{Maximal coupling procedure and stability of discrete {Markov} chains.
 ii}.
\bjtitle{Theory Probab. Math. Stat.}
\bvolume{87},
\bfpage{65}--\blpage{78}
(\byear{2013})
\bid{doi={10.1090/S0094-9000-2014-00905-9}, mr={3241447}}
\end{barticle}
%
\OrigBibText
\begin{barticle}
\bauthor{\bsnm{Kartashov}, \binits{M.}},
\bauthor{\bsnm{Golomoziy}, \binits{V.}}:
\batitle{Maximal coupling procedure and stability of discrete {Markov} chains.
 ii}.
\bjtitle{Theory of probability and mathematical statistics}
\bvolume{87},
\bfpage{65}--\blpage{78}
(\byear{2013})
\end{barticle}
\endOrigBibText
\bptok{structpyb}%
\endbibitem

\bibitem{RenewalTheoryForFunctionals}
\begin{barticle}
\bauthor{\bsnm{Kluppelberg}, \binits{C.}},
\bauthor{\bsnm{S.}, \binits{P.}}:
\batitle{Renewal theory for functionals of a {Markov} chain with compact state
 space}.
\bjtitle{Ann. Probab.}
\bvolume{3},
\bfpage{2270}--\blpage{2300}
(\byear{2003})
\bid{doi={10.1214/\\aop/1068646385}, mr={2016619}}
\end{barticle}
%
\OrigBibText
\begin{barticle}
\bauthor{\bsnm{Kluppelberg}, \binits{C.}},
\bauthor{\bsnm{S.}, \binits{P.}}:
\batitle{Renewal theory for functionals of a {Markov} chain with compact state
 space}.
\bjtitle{The Annals of Probability}
\bvolume{3},
\bfpage{2270}--\blpage{2300}
(\byear{2003})
\end{barticle}
\endOrigBibText
\bptok{structpyb}%
\endbibitem

\bibitem{Lindvall}
\begin{bbook}
\bauthor{\bsnm{Lindvall}, \binits{T.}}:
\bbtitle{Lectures on Coupling Method}.
\bpublisher{John Wiley and Sons}
(\byear{1991})
\bid{mr={1180522}}
\end{bbook}
%
\OrigBibText
\begin{bbook}
\bauthor{\bsnm{Lindvall}, \binits{T.}}:
\bbtitle{Lectures on Coupling Method}.
\bpublisher{John Wiley and Sons}
(\byear{1991})
\end{bbook}
\endOrigBibText
\bptok{structpyb}%
\endbibitem

\bibitem{NonlinearMarkovRenewalTheory}
\begin{barticle}
\bauthor{\bsnm{Melfi}, \binits{V.}}:
\batitle{Nonlinear {Markov} renewal theory with statistical applications}.
\bjtitle{Ann. Probab.}
\bvolume{20},
\bfpage{753}--\blpage{771}
(\byear{1992})
\bid{mr={1159572}}
\end{barticle}
%
\OrigBibText
\begin{barticle}
\bauthor{\bsnm{Melfi}, \binits{V.}}:
\batitle{Nonlinear {Markov} renewal theory with statistical applications}.
\bjtitle{Annals of Probability}
\bvolume{20},
\bfpage{753}--\blpage{771}
(\byear{1992})
\end{barticle}
\endOrigBibText
\bptok{structpyb}%
\endbibitem

\bibitem{Thorisson}
\begin{bbook}
\bauthor{\bsnm{Thorisson}, \binits{H.}}:
\bbtitle{Coupling, Stationarity, and Regeneration}.
\bpublisher{Springer},
\blocation{New York}
(\byear{2000})
\bid{doi={10.1007/978-1-4612-1236-2}, mr={1741181}}
\end{bbook}
%
\OrigBibText
\begin{bbook}
\bauthor{\bsnm{Thorisson}, \binits{H.}}:
\bbtitle{Coupling, Stationarity, and Regeneration}.
\bpublisher{Springer},
\blocation{New York}
(\byear{2000})
\end{bbook}
\endOrigBibText
\bptok{structpyb}%
\endbibitem

\end{thebibliography}
\end{document}